\documentclass[12pt]{article}

\usepackage[colorlinks=true, pdfstartview=FitV, linkcolor=blue, citecolor=blue, urlcolor=blue]{hyperref}

\usepackage{amssymb,amsmath, amscd}
\usepackage{times, verbatim}
\usepackage{graphicx}

\newtheorem{theorem}{Theorem}
\newtheorem{prop}[theorem]{Proposition}
\newtheorem{proposition}[theorem]{Proposition}

\newtheorem{lemma}[theorem]{Lemma}

\newtheorem{remark}[theorem]{Remark}
\newenvironment{proof}{\noindent {\bf Proof:}}{$\Box$ \vspace{2 ex}}

\def\Z{{\mathbb Z}}
\def\Q{{\mathbb Q}}

\DeclareMathOperator{\End}{End}

\DeclareMathOperator{\SO}{SO}

\DeclareMathOperator{\GL}{GL}
\DeclareMathOperator{\SL}{SL}
\DeclareMathOperator{\Sp}{Sp}
\DeclareMathOperator{\PGL}{PGL}
\DeclareMathOperator{\PSL}{PSL}
\DeclareMathOperator{\PGSp}{PGSp}
\DeclareMathOperator{\Trace}{Trace}
\DeclareMathOperator{\disc}{disc}
\DeclareMathOperator{\Res}{Res}
\DeclareMathOperator{\Spec}{Spec}
\DeclareMathOperator{\Sel}{Sel}
\DeclareMathOperator{\Pic}{Pic}
\DeclareMathOperator{\Sym}{Sym}

\title{Arithmetic invariant theory}

\author{Manjul Bhargava and Benedict H. Gross}

\begin{document}
\maketitle

\tableofcontents

\section{Introduction}

Let $k$ be a field, let $G$ be a reductive algebraic group over $k$, and let $V$ be a linear representation of~$G$. Geometric invariant theory involves the study of the $k$-algebra of $G$-invariant polynomials on~$V$, and the relation between these invariants and the $G$-orbits on $V$, usually under the hypothesis that the base field $k$ is algebraically closed. In favorable cases, one can determine the geometric quotient $V/\!\!/G = \Spec(\Sym^*(V^\vee))^G$ and can identify certain fibers of the morphism $V \rightarrow V/\!\!/G$ with certain $G$-orbits on $V$.

As an example, consider the three-dimensional adjoint representation of $G = \SL_2$ given by conjugation on the space $V$ of $2 \times 2$ matrices
$v = \bigl( \begin{smallmatrix}
a & b \\
c & -a 
 \end{smallmatrix} \bigr)$
% $v = \left( \begin{array}{cc}
%a & b \\
%c & -a \\
 %\end{array} \right)$
 of trace zero. This is irreducible when the characteristic of $k$ is not equal to $2$, which we assume here. It has the quadratic invariant $q(v) = -\det(v) = bc + a^2$, which generates the full ring of polynomial invariants. Hence $V/\!\!/G$ is isomorphic to the affine line and $q: V \rightarrow V/\!\!/G = \mathbb G_a$. If $v$ and $w$ are two vectors in $V$ with $q(v) = q(w) \neq 0$, then they lie in the same $G$-orbit provided that the field $k$ is separably closed. 

For general fields the situation is more complicated. In our example, let $d$ be a non-zero element of $k$ and let $K$ be the \'etale quadratic algebra $k[x]/ (x^2 - d)$. Then the $G(k)$-orbits on the set of vectors $v \in V$ with $q(v) = d \neq 0$ can be identified with elements in the $2$-group $k^*/NK^*$.  (See \S2.)

The additional complexity in the orbit picture, when $k$ is not separably closed, is what we refer to as arithmetic invariant theory. It can be reformulated using non-abelian Galois cohomology, but that does not give a complete resolution of the problem. Indeed, when the stabilizer $G_v$ of $v$ is smooth, we will see that there is a bijection between the different orbits over $k$ which lie in the orbit of $v$ over the separable closure and the elements in the kernel of the map in Galois cohomology $\gamma: H^1(k,G_v) \rightarrow H^1(k,G)$. Since $\gamma$ is only a map of pointed sets, the computation of this kernel can be non-trivial.

In this paper, we will illustrate some of the issues which remain by considering the {\it regular semi-simple} orbits---i.e., the closed orbits whose stabilizers have minimal dimension---in three representations of the split odd special orthogonal group $G = \SO_{2n+1} = \SO(W)$ over a field $k$ whose characteristic is not equal to $2$. Namely, we will study:
\begin{itemize}
\item
the standard representation $V = W$; 
\item
the adjoint representation $V = \frak{so}(W) = \wedge^2(W)$; and 
\item
the symmetric square representation $V = \Sym^2(W)$. 
\end{itemize}
In the first case, the map $\gamma$ is an injection and the arithmetic
invariant theory is completely determined by the geometric invariant
theory. In the second case, the stabilizer is a maximal torus and the
arithmetic invariant theory is the Lie algebra version of stable
conjugacy classes of regular semi-simple elements. The theory of stable
conjugacy classes, introduced by Langlands \cite{L}--\cite{L2} and developed further
by Shelstad \cite{Sh} and Kottwitz \cite{K1}, forms one of the key tools in the study of endoscopy and
the trace formula.
Here there are the
analogous problems, involving the Galois cohomology of tori, for the
adjoint representations of general reductive groups. In the third
case, there are {\it stable} orbits in the sense of Mumford's geometric
invariant theory \cite{Mu}, i.e., closed orbits whose stabilizers are
finite. Such representations arise more generally in Vinberg's invariant 
theory~(cf.~\cite{PV}, \cite{P}), where the torsion automorphism
corresponds to a regular elliptic class in the extended Weyl group. In
this case, we can use the geometry of pencils of quadrics to describe
an interesting subgroup of classes in the kernel of~$\gamma$.

Although we have focused here primarily on the case of orbits over a general field, a complete arithmetic invariant theory would also consider the orbits of a reductive group over more general rings such as the integers. We end with some remarks on integral orbits for the three representations we have discussed.

We would like to thank Brian Conrad, for his help with \'etale and flat cohomology, and Mark Reeder and Jiu-Kang Yu for introducing us to Vinberg's theory. We would also like to thank Bill Casselman, Wei Ho, Alison Miller, Jean-Pierre Serre, and the anonymous referee for a number of very useful comments on an earlier draft of this paper.  It is a pleasure to dedicate this paper to Nolan Wallach, who introduced one of us (BHG) to the beauties of invariant theory.

\section{Galois cohomology}

Let $k$ be a field, let $k^s$ be a separable closure of $k$, and let $k^a$ denote an algebraic closure containing $k^s$. Let $\Gamma$ be the (profinite) Galois group of $k^s$ over $k$. Let $G$ be a reductive group over $k$ and $V$ an algebraic representation of $G$ on a finite-dimensional $k$-vector space. The problem of classifying the $G(k)$-orbits on $V(k)$ which lie in a fixed $G(k^s)$-orbit can be translated (following Serre \cite[\S I.5]{S}) into the language of Galois cohomology.

Let $v \in V(k)$ be a fixed vector in this orbit, and let $G_v$ be the stabilizer of $v$.  We assume that $G_v$ is a smooth algebraic group over $k$. If $w \in V(k)$ is another vector in the same $G(k^s)$-orbit as $v$, then we may write $w = g(v)$ with $g \in G(k^s)$ well-defined up to right multiplication by $G_v(k^s)$. For every $\sigma \in \Gamma$, we have $g^{\sigma} = g a_{\sigma}$ with $a_{\sigma} \in G_v(k^s)$. The map $\sigma \rightarrow a_{\sigma}$ is a continuous $1$-cocycle  on $\Gamma$ with values in $G_v(k^s)$, whose class in the first cohomology set $H^1(\Gamma,G_v(k^s))$ is independent of the choice of $g$. Since $a_{\sigma} = g^{-1} g^{\sigma}$, this class is trivial when mapped to the cohomology set $H^1(\Gamma,G(k^s))$. We will use the notation $H^1(k,G_v)$ and $H^1(k,G)$ to denote these Galois cohomology sets in this paper.

Reversing the argument, one can show similarly that an element in the kernel of the map of pointed sets
$H^1(k, G_v) \rightarrow H^1(k,G)$ gives rise to a $G(k)$-orbit on $V(k)$ in the $G(k^s)$-orbit of $v$. Hence we obtain the following.

\begin{prop}\label{first}
There is a bijection between the set of $G(k)$-orbits on the vectors $w$ in $V(k)$ that lie in the same $G(k^s)$-orbit as $v$ and the kernel of the map
\begin{equation}
\gamma: H^1(k, G_v) \rightarrow H^1(k,G)
\end{equation}
in Galois cohomology.
 \end{prop}

When the stabilizer $G_v$ is smooth over $k$, the set of all vectors $w\in V(k)$ lying in the same $G(k^s)$-orbit as $v$ can be identified with the $k$-points of the quotient variety $G/G_v$, and the central problem of arithmetic invariant theory in this case is to understand the kernel of the map $\gamma$ in Galois cohomology. This is particularly interesting when $k$ is a finite, local, or global field, when the cohomology of the two groups $G_v$ and $G$ can frequently be computed.

In the example of the introduction with $G = \SL_2$ and $V$ the adjoint representation (again assuming char$(k) \neq 2$), let $v$ be a vector in $V(k)$ with $q(v) = d \neq 0$. Then the stabilizer $G_v$ is a maximal torus in $\SL_2$ which is split by the \'etale quadratic algebra $K$. The pointed set $H^1(k,G) = H^1(k,\SL_2)$ is trivial, so all classes in the abelian group $H^1(k,G_v) = k^*/NK^*$ lie in the kernel of $\gamma$. These classes index the orbits of $\SL_2(k)$ on the set $S$ of non-zero vectors $w$ with $q(w) = q(v)$, since this is precisely the set $S$ of vectors $w\in V(k)$ which lie in the same $\SL_2(k^s)$-orbit as $v$. (This illustrates the point that one first has to solve the orbit problem over the separable closure $k^s$, before using Proposition~\ref{first} to descend to orbits over $k$.)

The vanishing of $H^1(k,G)$ occurs whenever $G = \GL_n$ or $G = \SL_n$ or $G = \Sp_{2n}$, and gives an elegant solution to many orbit problems. For example, when the characteristic of $k$ is not equal to $2$, the classification of the non-degenerate orbits of $\SL_n = \SL(W)$ on the symmetric square representation $V = \Sym^2(W^\vee)$ shows that the isomorphism classes of non-degenerate orthogonal spaces $W$ of dimension $n$  over $k$ with a fixed determinant in $k^*/k^{*2}$ correspond bijectively to classes in $H^1(k,G_v) = H^1(k,\SO(W))$ (cf.~\cite[Ch VII, \S29]{K}, \cite[Ch III, Appendix~2, \S4]{S}).  In general, both $H^1(k,G_v)$ and $H^1(k,G)$ are non-trivial, and the determination of the kernel of $\gamma$ remains a challenging problem. 

\begin{remark}
{\em In those cases where the stabilizer $G_v$ is not smooth, it is at least flat of finite type over $k$, so one can replace the map $\gamma$ in Galois (\'etale) cohomology with one in flat (fppf) cohomology. Indeed, the $k$-valued points of $G/G_v$ can always be identified with the possibly larger set $S'$ of vectors $w'$ in $V(k)$ which lie in the same $G(k^a)$-orbit as $v$, where $k^a$ is an algebraic closure of $k^s$. As an example, the group $\mathbb G_m$ acts on $V = \mathbb G_a$ by the formula $\lambda(v) = \lambda^p\cdot v$. The stabilizer of $v = 1$ is the subgroup $\mu_p$, and $G/G_v = \mathbb G_m/\mu_p = \mathbb G_m$. The stabilizer $G_v$ is smooth if the characteristic of $k$ is not equal to $p$, in which case the set $S$ consists of the non-zero elements of the field $k$, and the $G(k)$-orbits on $S$ form a principal homogeneous space for the group $ H^1(k,\mu_p) = k^*/k^{*p}$. If the characteristic of $k$ is equal to $p$, the stabilizer $\mu_p$ is not smooth over $k$. In this case the set $S$ consists of the $p^{\rm th}$ powers in $k^*$. The set $S'$ is equal to the full group of non-zero elements in $k$, which is strictly larger than $S$ when the field $k$ is  imperfect. In the general case one can show that the $G(k)$ orbits on $S' = (G/G_v)(k)$ are in bijection with the kernel of the map $\gamma_f: H^1_f(k,G_v) \rightarrow H^1_f(k,G)$ in flat (fppf) cohomology. In our example, we get a bijection of these orbits with the flat cohomology group $H^1_f(k,\mu_p) = k^*/k^{*p}$, as $H^1_f(k,\mathbb G_m) = 1$. 

The semi-simple orbits in the three representations that we will study in this paper all have smooth stabilizers $G_v$. Hence we only consider the map $\gamma$ in Galois cohomology. }
\end{remark}

\section{Some representations of the split odd special orthogonal group}

Let $k$ be a field, with char$(k) \neq 2$. Let $n \geq 1$ and let $W$ be a fixed non-degenerate, split orthogonal space over $k$, of dimension $2n+1 \geq 3$ and determinant $(-1)^n$ in $k^*/k^{*2}$. Such an orthogonal space is unique up to isomorphism. If $\langle v,w\rangle$ is the bilinear form on $W$, then we may choose an ordered basis $\{e_1, e_2,\ldots,e_n, u, f_n, \ldots,f_2,f_1\}$ of $W$ over $k$ with inner products given by
\begin{equation}\label{dprods}
\begin{array}{c}
\langle e_i,e_j \rangle = \langle f_i,f_j \rangle = \langle e_i,u \rangle = \langle f_i,u \rangle =0,\\[.1in]
 \langle e_i, f_j \rangle = \delta_{ij},\\[.1in]
\langle u,u \rangle = 1.
\end{array}
\end{equation}
The Gram matrix of the bilinear form with respect to this basis (which we will call {\it the standard basis}) is an anti-diagonal matrix. (A good general reference on orthogonal spaces, which gives proofs of these results, is \cite{M}.)

Let $T: W \rightarrow W$ be a $k$-linear transformation. We define the {\it adjoint transformation} $T^*$ by the formula
$$\langle Tv,w \rangle = \langle v,T^*w \rangle.$$
The matrix of $T^*$ in our standard basis is obtained from the matrix of $T$ by reflection around the anti-diagonal. In particular, we have the identity $\det(T) = \det(T^*)$.
We say a linear transformation $g: W \rightarrow W$ is {\it orthogonal} if
$\langle gv,gw \rangle = \langle v,w \rangle.$
Then $g$ is invertible, with $g^{-1} = g^*$, and $\det(g) = \pm1$ in $k^*$. We define the {\it special orthogonal group} $\SO(W)$ of $W$ by
\begin{equation}\label{orthdef}
\SO(W) := \{g \in \GL(W):  gg^* = g^*g = 1, ~ \det(g) = 1\}.
\end{equation}
We are going to consider the arithmetic invariant theory for three representations $V$ of the reductive group $G = \SO(W)$ over $k$. 

The first is the standard representation $V = W$, which is irreducible and symmetrically self-dual 
(isomorphic to its dual by a symmetric bilinear pairing)
of dimension~$2n+1$. Here we will see that the invariant polynomial $q_2(v) := \langle v,v \rangle$  generates the ring of polynomial invariants and separates the non-zero orbits over $k$. 

The second is the adjoint representation $V = \frak{so}(W)$, which is irreducible and symmetrically self-dual 
of dimension $2n^2 + n$. This representation is isomorphic to the exterior square $\wedge^2(W)$ of $W$, and can be realized as the space of skew self-adjoint operators:
\begin{equation}
V =  \wedge^2(W) = \{T: W \rightarrow W: T =- T^*\},
\end{equation}
where $g \in G$ acts by conjugation: $T \mapsto gTg^{-1} = gTg^*$. The Lie bracket on $V$ is given by the formula $[T_1,T_2] = T_1T_2 - T_2T_1$ and the duality by $\langle T_1,T_2 \rangle = \Trace(T_1T_2)$. Here the theory of $G(k)$-orbits in a fixed $G(k^s)$-orbit is the Lie algebra version of stable conjugacy classes for the group $G = \SO(W)$.

The third is a representation $V$ which arises in Vinberg's theory,
from an outer involution $\theta$ of the group $\GL(W)$. It is
isomorphic to the symmetric square $\Sym^2(W)$ of $W$, and can be
realized as the space of self-adjoint operators:
\begin{equation}
V = \Sym^2(W) = \{T: W \rightarrow W: T = T^*\},
\end{equation}
where again $G = \SO(W)$ acts by conjugation. This representation has dimension $2n^2 + 3n + 1$ and is symmetrically self-dual by the pairing $\langle T_1,T_2 \rangle = \Trace (T_1T_2).$ We will see that there are stable orbits, 
%in the sense of geometric invariant theory, 
and that the arithmetic invariant theory of the stable orbits involves the arithmetic of hyperelliptic curves of genus $n$ over $k$, with a $k$-rational Weierstrass point.

We note that the third representation $V$ is not irreducible, as it
contains the trivial subspace spanned by the identity matrix, and has
a non-trivial invariant linear form given by the trace. When the
characteristic of $k$ does not divide $2n+1 = \dim(W)$, the
representation $V$ is the direct sum of the trivial subspace and the
kernel of the trace map, and the latter is irreducible and
symmetrically self-dual of dimension $2n^2 + 3n$. When the
characteristic of $k$ divides $2n+1$ the trivial subspace is contained
in the kernel of the trace. In this case $V$ has two trivial factors
and an irreducible factor of dimension $2n^2 + 3n -1$ in its
composition series.

\section{Invariant polynomials and the discriminant}

In the standard representation $V=W$ of $G = \SO(W)$, the quadratic invariant $q_2(v) = \langle v,v \rangle$ generates the ring of invariant polynomials. We define $\Delta = q_2$ in this case. When $\Delta(v) \neq 0$, the stabilizer $G_v$ is the reductive subgroup $\SO(U)$, where $U$ is the hyperplane in $W$ of vectors orthogonal to $v$. 

In the second and third representations, the group $\SO(W)$ acts by conjugation on the subspace $V$ of $\End(W)$. Hence the characteristic polynomial of an operator $T$ is an invariant of the $G(k)$-orbit. 

For the adjoint representation, the operator $T$ is skew self-adjoint and its characteristic polynomial has the form
$$f(x) = \det(xI - T) =  x^{2n+1} + c_2x^{2n-1} + c_4x^{2n-3} + \cdots + c_{2n}x = xg(x^2).$$
with coefficients $c_{2m}\in k$. The coefficients $c_{2m}$ are polynomial invariants of the representation, with $\deg(c_{2m}) = 2m$.  These polynomials are algebraically independent and generate the full ring of polynomial invariants on $V = \frak{so}(W)$ over $k$ \cite[Ch 8, \S8.3, \S 13.2, VI] {B}.. An important polynomial invariant, of degree $2n(2n+1)$, is the discriminant $\Delta$ of the characteristic polynomial of $T$:
$$\Delta = \Delta(c_2,c_4,\ldots,c_{2n}) = \disc f(x). $$
This is non-zero in $k$ precisely when the polynomial $f(x)$ is separable, so has $2n+1$ distinct roots in the separable closure $k^s$ of $k$. The condition $\Delta(T) \neq 0$ defines the regular semi-simple orbits in the Lie algebra. For such an orbit, we will see that the stabilizer $G_T$ is a maximal torus in $G$, of dimension~$n$ over $k$. 

For the third representation $V$ on self-adjoint operators, the characteristic polynomial $f(x)$ of $T$ can be any monic polynomial of degree $2n+1$; we write
$$f(x) = \det(xI - T) =  x^{2n+1} + c_1x^{2n} + c_2x^{2n-1} + \cdots + c_{2n}x  + c_{2n+1}$$
with coefficients $c_m\in k$. Again the $c_m$ give algebraically independent polynomial invariants, with $\deg(c_m)=m$, which generate the full ring of polynomial invariants on $V$ over $k$. The discriminant 
$$\Delta = \Delta(c_1,c_2,\ldots,c_{2n+1}) = \disc f(x)$$
is defined as before, and is non-zero when $f(x)$ is separable. We will see that the condition $\Delta(T) \neq 0$ defines the stable orbits of $G$ on $V$. For such an orbit, we will see that the stabilizer $G_T$ is a finite commutative group scheme of order $2^{2n}$ over $k$, which embeds as a Jordan subgroup scheme of $G$ (see~\cite[Ch 3]{KT}).

%When $k = k^s$ is separably closed, we will see that the orbits in these representations with discriminant $\Delta \neq 0$ are closed, and are defined by the values of the invariant polynomials. This is the situation studied by geometric invariant theory.  We will then consider the orbits having non-zero discriminant over an arbitrary field $k$.

\section{The orbits with non-zero discriminant}

In this section, for each of the three representations $V$, we exhibit an orbit for $G$ where the invariant polynomials described above take arbitrary values in $k$, subject to the single restriction that $\Delta \neq 0$. We calculate the stabilizer $G_v$ and its cohomology $H^1(k, G_v)$ in terms of the values of the invariant polynomials on $v$. We also give an explicit description of the map $\gamma: H^1(k,G_v) \rightarrow H^1(k,G)$. We note that all three representations arise naturally in Vinberg's invariant theory, and the representative orbits that we will construct are in the Kostant section (cf. \cite{P}).

When $V= W$ is the standard representation, let $d$ be an element of $k^*$. The vector $v = e_1 + df_1$ has $q_2(v) = \Delta (v) = d$. The stabilizer $G_v$ acts on the orthogonal complement $U$ of the non-degnerate line $kv$ in $W$, which is a quasi-split orthogonal space of dimension $2n$ and discriminant $d$ in $k^*/k^{*2}$. (The {\it discriminant} of an orthogonal space of dimension $2n$ is defined as $(-1)^n$ times its determinant.) This gives an identification $G_v = \SO(U)$, where the special orthogonal group $\SO(U)$ is quasi-split over $k$ and split by $k(\sqrt d)$. Witt's extension theorem \cite[Ch 1]{M} shows that all vectors $w$ with $q_2(w) = d$ lie in the $G(k)$-orbit of $v$, so the invariant polynomials separate the orbits over $k$ with non-zero discriminant. One can also show that there is a single non-zero orbit with $q_2(v) = 0$, represented by the vector $v = e_1 = e_1 + 0f_1.$

The cohomology set $H^1(k,\SO(U))$ classifies non-degenerate orthogonal spaces $U'$ of dimension~$2n$ and discriminant $d$ over $k$, and the cohomology set $H^1(k, \SO(W))$ classifies non-degenerate orthogonal spaces $W'$ of dimension $2n+1$ and determinant
$(-1)^n$ over $k$, with the trivial class corresponding to the split space $W$. The map
$$\gamma: H^1(k,G_v) = H^1(k, \SO(U)) \longrightarrow H^1(k,G) = H^1(k, \SO(W))$$
is given explicitly by mapping the space $U'$ to the space $W' = U' + \langle d \rangle$. 
Witt's cancellation theorem~\cite{M} shows that the map~$\gamma$ is an injection of sets in this case, so the arithmetic invariant theory for the standard representation of any odd orthogonal group is the same as its geometric invariant theory.

\vspace{.15in}
For the second representation $V = \frak{so}(W) = \wedge^2(W)$, let $$f(x) = x^{2n+1} + c_2x^{2n-1} + c_4x^{2n-3} + \cdots + c_{2n}x$$ be a polynomial in $k[x]$ with non-zero discriminant. We will construct a skew self-adjoint operator~$T$ on $W$ with characteristic polynomial $f(x)$. Since $f(x) = xh(x) = xg(x^2)$, we have $$\disc f(x) = c_{2n}^2 \disc h(x) = (-4)^n c_{2n}^3 \disc g(x)^2.$$
 Let $K = k[x] /(g(x))$, $E = k[x]/(h(x))$, and $L = k[x]/(f(x))$. By our assumption that $\Delta \neq 0$, these are \'etale $k$-algebras of ranks $n$, $2n$, and $2n+1$ respectively. We have $L = E \oplus k$. Furthermore the map $x \to -x$ induces an involution $\tau$ of the algebras $E$ and of $L$, with fixed algebras $K$ and $K \oplus k$ respectively. 

Let $\beta$ be the image of $x$ in $L = k[x]/(f(x))$, so $f(\beta) = 0$ in $L$ and $f'(\beta)$ is a unit in $L^*$, We define a symmetric bilinear form $\langle \;,\: \rangle$  on the $k$-vector space $L = k + k\beta + k\beta^2 + \cdots+k\beta^{2n}$ by taking 
\begin{equation}
\langle \lambda, \mu \rangle :=  \mbox{ the coefficient of $\beta^{2n}$ in the product $(-1)^n\lambda\mu^{\tau}$.}
\end{equation}
This is non-degenerate, of determinant~$(-1)^n$, and the map $t(\lambda) = \beta\lambda$ is skew self-adjoint, with characteristic polynomial~$f(x)$. Finally, the subspace $M = k + k\beta + \cdots + k\beta^{n-1}$ is isotropic of dimension~$n$, so the orthogonal space~$L$ is split and isomorphic to~$W$ over~$k$. Choosing an isometry $\theta: L \rightarrow W$ we obtain a skew self-adjoint operator $T = \theta t\theta^{-1}$ on $W$ with the desired separable characteristic polynomial. Since
the isometry $\theta$ is unique up to composition with an orthogonal transformation of~$W$, the orbit of~$T$ is well-defined. The stabilizer of $T$ in ${\rm O}(W)$ has $k$-points $\{\lambda \in L^*: \lambda^{1 + \tau} = 1\}$. The subgroup $G_T$ which fixes $T$ is a maximal torus in $G = \SO(W)$, isomorphic to the torus $\Res_{K/k} U_1(E/K)$ of dimension $n$ over $k$.

Over the separable closure $k^s$ of $k$, any skew self-adjoint operator $S$ with (separable) characteristic polynomial $f(x)$ is in the same orbit of $T$. Indeed, since $f(x)$ is separable, it is also the minimal polynomial of $T$ and $S$, so we can find an element $g$ in $\GL(W)$ with $S = gTg^{-1}$. Since both operators are skew self-adjoint, the product $g^*g$ is in the centralizer of $T$ in $\GL(W)$. The centralizer of $T$ in $\End(W)$ is the algebra $k[T] = L$. Since  $g^*g$ is self-adjoint in $L^*$, and its determinant is a square in $k^*$, we see that $g^*g$ is an element of the subgroup $K^* \times k^{*2}$. Over the separable closure, every element of $K^* \times k^{*2}$ is a norm from $L^*$: $g^*g = h^{1+\tau}$. Then $gh^{-1}$ is an orthogonal transformation of $W$ over $k^s$ mapping $T$ to $S$. Hence $S$ is in the $\SO(W)(k^s)$-orbit of $T$.

To understand the orbits with a fixed separable characteristic polynomial over $k$, we need an explicit form of the map $\gamma$ in Galois cohomology. Since the stabilizer of $T$ is abelian, the pointed set $H^1(k,G_T)$ is an abelian group, which is isomorphic to $K^*/NE^*$ by Hilbert's Theorem $90$. The map 
$$\gamma : K^*/NE^* = H^1(k,G_T) \longrightarrow H^1(k,G) = H^1(k, \SO(W))$$
is given explicitly as follows. We first associate to an element $\kappa\in K^*$ the element $\alpha = (\kappa, 1)$ in $(L^{\tau})^* = K^*\times k^*$, with square norm from $L^*$ to $k^*$. We then associate to $\alpha$ the vector space $L$ with symmetric bilinear form
\begin{equation}
\langle \lambda, \mu \rangle_{\alpha} :=  \mbox{the coefficient of $\beta^{2n}$ in the product $(-1)^n\alpha\lambda\mu^{\tau}$.}
\end{equation}
 This orthogonal space $W_\kappa$ has dimension $2n+1$ and determinant $(-1)^n$ over $k$, and its isomorphism class depends only on the class of $\kappa$ in the quotient group $K^*/NE^* = H^1(k,G_T)$.

\begin{lemma}\label{co}
The orthogonal space $W_\kappa$ represents the class $\gamma(\kappa)$ in $H^1(k,\SO(W))$.
\end{lemma}

\begin{proof}
We first recall the recipe for associating to a cocycle $g_{\sigma}$ on the Galois group with values in $\SO(W)(k^s)$ a new orthogonal space $W'$ over $k$. We use the inclusion $\SO(W) \rightarrow \GL(W)$ and the triviality of  $H^1(k,\GL(W))$ to write $g_{\sigma} = h^{-1}h^{\sigma}$ for an element $h \in \GL(W)(k^s)$. We then define a new non-degenerate symmetric bilinear form on $W$ by the formula
\begin{equation}\label{newprod}
\langle v,w \rangle^* = \langle h^{-1}v, h^{-1}w \rangle.
\end{equation}
This takes values in $k$ and defines the space $W'$, which has dimension $2n+1$ and determinant $(-1)^n$. The isomorphism class of $W'$ over $k$ depends only on the cohomology class of the cocycle $g_{\sigma}$ in $H^1(k,\SO(W))$.

In our case, the cocycle $g_{\sigma}$ representing $\gamma(\kappa)$ comes from a cocycle with values in the stabilizer $G_v$. This is a maximal torus in $\SO(W)$, which is a subgroup of the maximal torus $\Res_{L/k}\mathbb G_m$ of $\GL(W)$. This torus already has trivial Galois cohomology, so we can write $g_{\sigma} = h^{\sigma}/h$ with $h \in (L \otimes k^s)^*$ satisfying $h^{1 + \tau} = \alpha$. Substituting this particular $h$ into formula (\ref{newprod}) for the new inner product on $W$ completes the proof. 
\end{proof}

We note that the class $\kappa$ above will be in the kernel of $\gamma$ precisely when the quadratic space $W'$ with bilinear form $\langle\;,\,\rangle_{\alpha}$ is split. Such classes give additional orbits of $\SO(W)$ on ${\frak{so}}(W) = \wedge^2(W)$ over $k$ with characteristic polynomial $f(x)$.

\vspace{.15in}
The analysis for the third representation $V = \Sym^2(W)$ is similar. Here we start with an arbitrary monic separable polynomial $f(x) = x^{2n+1} + c_1x^{2n} + \cdots + c_{2n+1}$ and wish to construct a self-adjoint operator $T$ on $W$ with characteristic polynomial $f(x)$. We let $L = k[x]/(f(x))$, which is an \'etale $k$-algebra of rank $2n+1$, and let $\beta$ be the image of $x$ in $L$.  We define a symmetric bilinear form $\langle \lambda, \mu \rangle$ on $L = k + k\beta + \cdots + k\beta^{2n}$ by taking the coefficient of $\beta^{2n}$ in the product $\lambda\mu$. This is non-degenerate of determinant $(-1)^n$, and the map $t(\lambda) = \beta\lambda$ is self-adjoint, with characteristic polynomial $f(x)$. Finally, the subspace $M = k + k\beta + \cdots + k\beta^{n-1}$ is isotropic of dimension $n$, so the orthogonal space $L$ is split and isomorphic to $W$ over $k$. Choosing an isometry $\theta: L \rightarrow W$, we obtain a self-adjoint operator $T = \theta t\theta^{-1}$ on $W$ with the desired separable characteristic polynomial. Since
the isometry $\theta$ is unique up to composition with an orthogonal transformation of $W$, the orbit of $T$ is well-defined. The stabilizer of $T$ in $\O(W)$ has $k$-points $\{\lambda \in L^*: \lambda^2 = 1\}$. The subgroup $G_T$ in $\SO(W)$ which fixes $T$ is the finite \'etale group scheme $A$ of order $2^{2n}$, which is the kernel of the norm map $\Res_{L/k} (\mu_2) \rightarrow \mu_2$. 

Over the separable closure $k^s$ of $k$, any self-adjoint operator $S$ with (separable) characteristic polynomial $f(x)$ is in the same orbit as $T$. Indeed, since $f(x)$ is separable, it is also the minimal polynomial of $T$ and $S$, so we can find an element $g\in\GL(W)$ with $S = gTg^{-1}$. Since both operators are self-adjoint, the product $g^*g$ is in the centralizer of $T$ in $\GL(W)$. The centralizer of $T$ in $\End(W)$ is the algebra $k[T] = L$, so $g^*g$ is an element of $L^*$. Over the separable closure, every element of $L^*$ is a square: $g^*g = h^2$. Then $gh^{-1}$ is an orthogonal transformation of $W$ over $k^s$ mapping $T$ to $S$. Hence $S$ is in the $\SO(W)(k^s)$-orbit of $T$.

We now consider the orbits with a fixed separable characteristic polynomial over $k$. Since the stabilizer of $T$ is again abelian, the pointed set $H^1(k,G_T)$ is an abelian group which is isomorphic to $(L^*/L^{*2})_{N=1}$ by Kummer theory. The map 
$$\gamma : H^1(k,G_T) = (L^*/L^{*2})_{N=1}  \longrightarrow H^1(k,G) = H^1(k, \SO(W))$$
is given explicitly as follows. We associate to an element $\alpha$ in $(L^*)_{N=1}$ the orthogonal space $L$ with bilinear form $\langle \lambda, \mu \rangle_{\alpha}$ given by the coefficient of $\beta^{2n}$ in the product $\alpha\lambda\mu$. This orthogonal space has dimension $2n+1$ and determinant $(-1)^n$ over $k$. Its isomorphism class over $k$ depends only on the image of $\alpha$ in the quotient group $(L^*/L^{*2})_{N=1} = H^1(k,G_T)$. This orthogonal space represents the class $\gamma(\alpha)$ in $H^1(k,\SO(W))$. The proof is the same as that of Lemma~\ref{co}. We first observe that the map taking the cocycle $g_{\sigma}$ from $G_T$ to $\SO(W)$ to $\GL(W)$ can also be obtained by mapping $G_T$ to the maximal torus $\Res_{L/k} \mathbb G_m$ in $\GL(W)$. This torus has trivial cohomology, so $\alpha = h^2$ with $h \in (L\otimes k^s)^*$, and this choice of $h$ gives the inner product $\langle \lambda, \mu \rangle_{\alpha}$. The class $\alpha$ will be in the kernel of $\gamma$ precisely when the quadratic space  $L$ with bilinear form $\langle\;,\,\rangle_{\alpha}$ is split; such classes give additional orbits of $\SO(W)$ on $\Sym^2(W)$ over $k$ with characteristic polynomial $f(x)$.

We summarize what we have established for the representations $V = {\frak{so}}(W)$ and $V = \Sym^2(W)$.

\begin{prop}\label{rationalws}
For each monic separable polynomial $f(x)$ of degree $2n+1$ over $k$ of the form $f(x) = xg(x^2)$ there is a distinguished $\SO(W)(k)$-orbit of skew self-adjoint operators $T$ on $W$ with characteristic polynomial $f(x)$. All other orbits on $\wedge^2(W)$ with this characteristic polynomial lie in the $\SO(W)(k^s)$-orbit of $T$, and correspond bijectively to the non-identity classes in the kernel of $\gamma: K^*/NE^* \rightarrow H^1(k,\SO(W)),$ where $K=k[x]/(g(x))$ and $E=k[x]/(g(x^2))$. 

For each monic separable polynomial $f(x)$ of degree $2n+1$ over $k$ there is a distinguished $\SO(W)(k)$-orbit of self-adjoint operators $T$ on $W$ with characteristic polynomial $f(x)$. All other orbits on $\Sym^2(W)$ with this characteristic polynomial lie in the $\SO(W)(k^s)$-orbit of $T$, and correspond bijectively to the non-identity classes in the kernel of $\gamma: (L^*/L^{*2})_{N=1} \rightarrow H^1(k,\SO(W)),$
where $L=k[x]/(f(x))$.
\end{prop}

\section{Stable orbits and hyperelliptic curves}

For both representations $V = \wedge^2(W)$ and $V = \Sym^2(W)$ of $G = \SO(W)$ we associated to the distinguished orbit $T$ with separable characteristic polynomial $f(x)$ and any class $\alpha$ in the cohomology group $H^1(k,G_T)$ a symmetric bilinear form $\langle \lambda, \mu \rangle_{\alpha}$ on the $k$-vector space $L = k[x]/ (f(x))$. The class $\alpha$ is in the kernel of the map $\gamma: H^1(k, G_T) \rightarrow H^1(k,G)$ precisely when this quadratic space is split over $k$. However, exhibiting specific classes $\alpha \neq 1$ where this space is split is a difficult general problem, so it is difficult to exhibit other orbits with this characteristic polynomial.

In the case of the third representation $V = \Sym^2(W)$, the orbits $T$ with $\Delta(T) \neq 0$ are stable; namely, %in the sense of geometric invariant theory. 
they are closed (defined by the values of the invariant polynomials over the separable closure) and have finite stabilizer (the commutative group scheme $A = \Res_{L/k} (\mu_2)_{N = 1}$ of order $2^{2n}$). In this case, we will use some 
results in algebraic geometry, on hyperelliptic curves with a Weierstrass point and the Fano variety of the complete intersection of two quadrics in $\mathbb P(L \oplus k)$, to produce certain classes in the kernel of the map $\gamma: H^1(k,A) \rightarrow H^1(k,\SO(W))$.

Let $C$ be the smooth projective hyperelliptic curve of genus $n$ over $k$ with affine equation
$y^2 = f(x)$ and $k$-rational Weierstrass point $P$ above $x = \infty$. The functions on $C$ which are regular outside of $P$ form an integral domain:
$$H^0(C - P, O_{C - P}) = k[x,y]/ (y^2 = f(x)) = k[x, \sqrt {f(x)}].$$

The complete curve $C$ is covered by this affine open subset $U_1$, together with the affine open
subset $U_2$ associated to the equation $w^2 = v^{2n+2}f(1/v)$ and containing the point $P = (0,0)$. The gluing of $U_1$ and $U_2$ is by $(v,w) = (1/x, y/x^{n+1})$ and $(x,y) = (1/v,w/v^{n+1})$ wherever these maps are defined.
Let $J$ denote the Jacobian of $C$ over $k$ and let $J[2]$ the kernel of multiplication by $2$ on $J$. This is a finite \'etale group scheme of order $2^{2n}$ over $k$.

\begin{lemma}\label{lemm}
The group scheme $J[2]$ of $2$-torsion on the Jacobian of $C$ is canonically isomorphic to the stabilizer $A = \Res_{L/k}(\mu_2)_{N=1}$ of the orbit $T$ in $\SO(W)$.
\end{lemma}

\begin{proof}
Write $L = k[x] / (f(x)) = k + k\beta + \cdots + k\beta^{2n}$, where $f(\beta) = 0$.
The other Weierstrass points $P_{\eta} = (\eta(\beta), 0)$ of $C(k^s)$ correspond bijectively to algebra embeddings $\eta: L \rightarrow k^s$. Associated to such a point we have the divisor $d_{\eta} = (P_{\eta}) - (P)$ of degree zero. The divisor class of $d_{\eta}$ lies in the $2$-torsion subgroup
$J[2](k^s)$ of the Jacobian, as
$$2d_{\eta} = {\rm div} (x - \eta(\beta)).$$
The Riemann-Roch theorem shows that the classes $d_{\eta}$ generate the finite group $J[2](k^s)$, and satisfy the single relation
$$\sum(d_{\eta}) = {\rm div}(y).$$
Since the Galois group of $k^s$ acts on these classes by permutation of the embeddings $\eta$, we have an isomorphism of group schemes: $J[2] \cong \Res_{L/k}(\mu_2)/\mu_2$. This quotient of $\Res_{L/k}(\mu_2)$ is isomorphic to the subgroup scheme $A = \Res_{L/k}(\mu_2)_{N=1}$, as the degree of $L$ over $k$ is odd. This completes the proof. 
\end{proof}

The exact sequence of Galois modules,
$$0 \rightarrow J[2](k^s) \rightarrow J(k^s) \rightarrow J(k^s) \rightarrow 0,$$
gives an exact descent sequence 
$$0 \rightarrow J(k)/2J(k) \rightarrow H^1(k, J[2]) \rightarrow H^1(k,J)[2] \rightarrow 0$$
in Galois cohomology.
By Lemma~\ref{lemm}, the middle term in this sequence can be identified with the group $H^1(k,A) = H^1(k, G_T)$, and our main result in this section is the following.

\begin{prop}
The subgroup $J(k)/2J(k)$ of  $H^1(k,A) = H^1(k,G_T)$ lies in the kernel of the map $\gamma: H^1(k,G_T) \rightarrow H^1(k,G)$.
\end{prop}

\begin{proof}
We first make the descent map from $H^1(k,A)$ to $H^1(k,J)[2]$ more explicit. That is, we need to associate to a class $\alpha$ in the group
$$H^1(k,A) = (L^*/L^{*2})_{N=1}$$
a principal homogeneous space $F_{\alpha}$ of order $2$ for the Jacobian $J$ over $k$. The class $\alpha$ will be in the subgroup $J(k)/2J(k)$ precisely when the homogeneous space $F_{\alpha}$ has a $k$-rational point.

We have previously associated to the class $\alpha$ the orthogonal space $L$ with symmetric bilinear form $\langle \lambda,\mu \rangle_{\alpha} :=$ the coefficient of $\beta^{2n}$ in the product $\alpha\lambda\mu$. We also defined a self-adjoint operator given by multiplication by $\beta$ on $L$, and that gives a second symmetric bilinear form on $L$: $\langle \beta\lambda,\mu \rangle_{\alpha} = \langle \lambda,\beta\mu \rangle_{\alpha}$. 

Let $M = L \oplus k$, which has dimension $2n+2$ over $k$, and consider the two quadrics on $M$ given by 
\begin{eqnarray*}
Q(\lambda,a) &=& \langle \lambda,\lambda \rangle_{\alpha}\\
Q'(\lambda,a) &=& \langle \beta\lambda, \lambda \rangle_{\alpha} + a^2.
\end{eqnarray*}
The pencil $uQ - vQ'$ is non-degenerate and contains exactly $2n+2$ singular elements over $k^s$, namely, the quadric $Q$ at $v=0$ and the $2n+1$ quadrics $\eta(\beta) Q - Q'$ at the points where $f(\eta(\beta)) = 0$. Hence the base locus is non-singular in $\mathbb P(M)$ and the Fano variety $F_{\alpha}$ of this complete intersection, consisting of the $n$-dimensional subspaces $Z$ of $M$ which are isotropic for all of the quadrics in the pencil, is a principal homogeneous space of order $2$ for the Jacobian $J$ (c.f.\ \cite{D}). More precisely, there is a commutative algebraic group $I_{\alpha}$ with $2$ components over $k$, having identity component $J$ and non-identity component $F_\alpha$.

Since the discriminant of the quadric $uQ - vQ'$ in the pencil is equal to $v^{2n+2}f(x)$ with $x = u/v$, a point $c = (x,y)$ on the hyperelliptic curve $y^2 = f(x)$ determines both a quadric $Q_x = xQ - Q'$ in the pencil together with a {\it ruling} of $Q_x$, i.e., a component of the variety of $(n+1)$-dimensional $Q_x$-isotropic subspaces in $M$. Each point gives an involution of the corresponding Fano variety $\theta(c): F_{\alpha} \rightarrow F_{\alpha}$ with $2^{2n}$ fixed points over a separable closure $k^s$ of $k$. The involution $\theta(c)$ is defined as follows. A point of $F_{\alpha}$ consist of a common isotropic subspace $Z$ of dimension $n$ in $M\otimes k^s$. The point $c$ gives a maximal isotropic subspace $Y$ for the quadric $Q_x$ which contains $Z$. If we restrict any non-singular quadric in the pencil (other than $Q_x$)  to $Y$, we get a reducible quadric which is the sum of two hyperplanes: $Z$ and another common isotropic subspace $Z'$. This defines the involution: $\theta(c)(Z) = Z'$.  In the algebraic group $I_\alpha$, we have that $Z+Z'$ is the class of the divisor $(c)-(P)$ of degree zero in $J$.

Now assume that the class $\alpha$ is in the subgroup $J(k)/2J(k)$. Then its image in $H^1(k,J)$ is trivial, and the homogenous space $F_{\alpha}$ has a $k$-rational point. Hence there is a $k$-subspace $Z$ of $M = L \oplus k$ which is isotropic for both $Q$ and $Q'$. Since it is isotropic for $Q'$, the subspace $Z$ does not contain the line $0 \oplus k$, so its projection to the subspace $L$ has dimension $n$ and is isotropic for $Q$. This implies that the orthogonal space $L$ with bilinear form $(\lambda,\nu)_{\alpha}$ is split, so the class $\alpha$ is in the kernel of the map $\gamma: H^1(k,A) \rightarrow H^1(k, \SO(W))$.
\end{proof}

Note that when $c = P$, the Weierstrass point over $x = \infty$, the involution $\theta(P)$  is induced by the linear involution $(\lambda,a) \rightarrow (\lambda, -a)$ of $M = L \oplus k$. The fixed points are just the $n$-dimensional subspaces $X$ over $k^s$ which are isotropic for both quadrics 
\begin{eqnarray*}
q(\lambda) &=& \langle \lambda,\lambda \rangle_{\alpha}\,,\\
q'(\lambda) &=& \langle \beta\lambda, \lambda \rangle_{\alpha}
\end{eqnarray*}
on the space $L$ of dimension $2n+1$ over $k$. There are $2^{2n}$ such isotropic subspaces over $k^s$, and they form a principal homogeneous space for $J[2]$. The variety $F_{\alpha}$ has a $k$-rational point when $\alpha$ lies in the subgroup $J(k)/2J(k)$, but only has a $k$-rational point fixed by the involution $\theta(P)$ when $\alpha$ is the trivial class in $H^1(k,J[2])$.

\begin{remark}{\em 
The finite group scheme $A = J[2]$ does not determine the hyperelliptic curve $C$ over $k$. Indeed, for any class $d \in k^*/k^{*2}$, the hyperelliptic curve $C_d$ with affine equation $dy^2 = f(x)$ has the same $2$-torsion subgroup of its Jacobian. This Jacobian $J_d$ of $C_d$ acts on the Fano variety of the complete intersection of the two quadrics given by
\begin{eqnarray*}
Q(\lambda,a) &=& \langle \lambda,\lambda \rangle_{\alpha}\,,\\
Q'(\lambda,a) &=& \langle \beta\lambda, \lambda \rangle_{\alpha} + da^2.
\end{eqnarray*}
Indeed, the discriminant of the quadric $uQ - vQ'$ in the pencil is equal to $dv^{2n+2} f(x)$, where $x = u/v$. A similar argument then shows that the subgroup $J_d(k)/2J_d(k)$ is also contained in the kernel of the map $\gamma$ on $H^1(k,A)$. }
\end{remark}

\section{Arithmetic fields}

In this section, we describe the orbits in our three representations when $k$ is a finite, local, or global field.

\subsection{Finite fields}

First, we consider the case when $k$ is finite, of odd order $q$. In this case, $H^1(k,\SO(W)) = 1$ by Lang's theorem, as $\SO(W)$ is connected. As a consequence, every quadratic space of dimension $2n+1$ and determinant $(-1)^n$ is split, and all elements of $H^1(k,G_T)$ lie in the kernel of $\gamma$. 

In the standard representation $V = W$ the stabilizer of a vector $v$ with $q_2(v) \neq 0$ is the connected
orthogonal subgroup $\SO(U)$, which also has trivial first cohomology. So for every non-zero element $d$ in $k^*$, there is a unique orbit of vectors with $q_2(v) = d$.  (We have already seen this for general fields via Witt's extension theorem.)

In the adjoint representation $V = \frak{so}(W)$, the stabilizer of  a vector $T$ with $\Delta(T) \neq 0$ is the connected torus $\Res_{K/k} U_1(E/K)$, which also has trivial first cohomology. So for each separable characteristic polynomial of the form $f(x) = xg(x^2)$ there is a unique orbit of skew self-adjoint operators $T$ with characteristic polynomial $f(x)$.

In the representation $V = \Sym^2(W)$ the stabilizer of $T$ with characteristic polynomial $f(x)$ satisfying $\disc(f) = \Delta(T) \neq 0$ is the finite group scheme $A = (\Res_{L/k} \mu_2)_{N=1}$. In this case $H^1(k, A) = (L^*/L^{*2})_{N=1}$ is an elementary abelian $2$-group of order $2^m$, where $m+1$ is the number of irreducible factors of $f(x)$ in $k[x]$. So $2^m$ is the number of distinct orbits with characteristic polynomial $f(x)$. But this is also the order of the stabilizer $H^0(k,A) = A(k) = (L^*[2])_{N=1}$ of any point in the orbit. Hence the number of self-adjoint operators $T$ with any fixed separable polynomial is equal to the order of the finite group $\SO(W)(q)$. This is given by the formula
$$\#\SO(W)(q) = q^{n^2}(q^{2n} -1)(q^{2n-2} - 1)\cdots(q^2 - 1).$$

By Lang's theorem, we also have $H^1(k,J) = 0$, where $J$ is the Jacobian of the smooth hyperelliptic curve $y^2 = f(x)$ of genus $n$ over $k$.
Hence the homomorphism $J(k)/2J(k) \rightarrow H^1(k,A)$ is an isomorphism and every orbit with characteristic polynomial $f(x)$ comes from a $k$-rational point on the Jacobian.

\subsection{Non-archimedean local fields}

Next, we consider the case when $k$ is a non-archimedean local field, with ring of integers $O$ and finite residue field $O/\pi O$ of odd order. In this case, Kneser's theorem on the vanishing of $H^1$ for simply-connected groups (cf. \cite[Th.\ 6.4]{PR}, \cite{S}) gives an isomorphism
$$H^1(k,\SO(W)) \cong H^2(k, \mu_2) \cong (\mathbb Z/2\mathbb Z).$$

For the standard representation $V=W$, we also have $H^1(k,G_v) = H^1(k,\SO(U)) \cong  (\mathbb Z/2\mathbb Z)$, except in the case when $\dim(V) = 3$ and $q_2(v) = 1$, when $\SO(U)$ is a split torus and $H^1(k,\SO(U)) = 1$. The map $\gamma$ is a bijection except in the special case.

For the adjoint representation $V = \frak{so}(W)$, Kottwitz has shown in the local case that the map
$$\gamma: H^1(k,G_v) = (K^*/NE^*) \rightarrow H^1(k,G) =  (\mathbb Z/2\mathbb Z)$$
is actually a homomorphism of groups \cite{K1}. Let $f(x) = xg(x^2)$, so $K = k[x]/(g(x))$ and $E = k[x]/(g(x^2))$. It follows from local class field theory that the group $K^*/NE^*$ is elementary abelian of order $2^m$, where $m$ is the number of irreducible factors $g_i(x)$ of $g(x)$ such that $g_i(x^2)$ remains irreducible over $k$. Kottwitz also shows that that the map $\gamma$ is surjective when $m \geq 1$. Hence the number of orbits with separable characteristic polynomial $f(x)$ is $1$ when $m = 0$, and is $2^{m-1}$ when $ m \geq 1$.

For the third representation $V = \Sym^2(W)$, the map
$$\gamma: H^1(k,A) = H^1(k,J[2]) \rightarrow H^2(k, \mu_2) \cong (\mathbb Z/2\mathbb Z)$$ is an even quadratic form. The associated bilinear form is the cup product on $H^1(k,J[2])$ induced from the Weil pairing $J[2] \times J[2] \rightarrow \mu_2$, and  $J(k)/2J(k)$ is a maximal isotropic subspace on which $\gamma = 0$. This allows us to count the number of stable orbits with a fixed characteristic polynomial.

Let $m+1$ be the number of irreducible factors of $f(x)$ in $k[x]$, and let $O_L$ be the integral closure of the ring $O$ in $L$. Then $H^1(k, A) = (L^*/L^{*2})_{N=1}$ has order $2^{2m}$ and the number of stable orbits with characteristic polynomial $f(x)$ is equal $2^{m-1}(2^m + 1) = 2^{2m-1} + 2^{m-1}$. The subgroup $J(k)/2J(k)$ has order $2^m$, which is also the order of the subgroup $(O_L^*/O_L^{*2})_{N=1}$ of units. These two subgroups coincide when the polynomial $f(x)$ has coefficients in $O$ and the quotient algebra $O[x] / (f(x))$ is maximal in $L$.
% (cf. \cite{Stoll} ).

\subsection{The local field $\mathbb R$}

We next consider the orbits in our representations when $ k = \mathbb R$ is the local field of real numbers. Then the pointed set $H^1(k,G) = H^1(k,\SO(W))$ has $n+1$ elements, corresponding to the quadratic spaces $W'$ of signature $(p,q)$ satisfying: $p+q = 2n+1$ and $q \equiv n$ ~(mod 2). The pointed set $H^1(k,G_v) = H^1(k,\SO(U))$ for the standard representation has $n+1$ elements when $q_2(v)$ has sign $(-1)^n$, and has $n$ elements when $q_2(v)$ has sign $-(-1)^n$. The map $\gamma$ is a bijection in the first case and an injection in the second case, when the definite quadratic space $W'$ does not have an orbit with $q_2(w^*) = q_2(v)$.

In the second and third representations, $H^1(k,G_T)$ is an elementary abelian $2$-group, and we will consider the situations where it has maximal rank. For the adjoint representation $V = \frak{so}(W)$, this occurs when all of the nonzero roots of the characteristic polynomial $f(x)$ of the skew self-adjoint transformation $T$ are purely imaginary. Thus $f(x) = xg(x^2)$ where $g(x)$ factors completely over the real numbers and all of its roots are strictly negative. In this case, the $2$-group $H^1(k,G_T) = K^*/NE^* = (\mathbb R^*)^n/N(\mathbb C^*)^n$ has rank $n$. The real orthogonal space $W$ decomposes into $n$ orthogonal $T$-stable planes and an orthogonal line on which $T = 0$. The signatures of these planes determine the real orbit of $T$. Writing $n = 2m$ or $n = 2m+1$, we see that there are $\binom{n}{m}$ elements in the kernel of $\gamma$. One can  show that $\gamma$ is surjective in this case, and calculate the order of each fiber as a binomial coefficient $\binom{n}{k}$.

For the symmetric square representation $V = \Sym^2(W)$, the $2$-group $H^1(k,G_T)$ has maximal rank when the characteristic polynomial $f(x)$ of the self-adjoint transformation $T$ factors completely over the real numbers. In this case, $H^1(k,G_T) = ((\mathbb R^*)^{2n+1}/(\mathbb R^{*2})^{2n+1})_{N=1}$ has rank $2n$. The real orthogonal space $W$ decomposes into $2n+1$ orthogonal eigenspaces for $T$, and the signatures of these lines determine the real orbit. Hence there are $\binom{2n+1}{n}$ elements in the kernel of $\gamma$. One can also show that $\gamma$ is surjective in this case, and calculate the order of each fiber as a binomial coefficient $\binom{2n+1}{k}$ with $k \equiv n$ ~(mod 2).

\subsection{Global fields}

Finally, we consider the representation $\Sym^2(V)$ when $k$ is a global field. In this case, the group $H^1(k,A) = H^1(k,J[2])$ is infinite. We will now prove that there are also infinitely many classes in the kernel of $\gamma$, so infinitely many orbits with characteristic polynomial $f(x)$. 
\begin{prop}
Every class $\alpha$ in the $2$-Selmer group $\Sel(J/k,2)$ of $H^1(k,J[2])$ lies in the kernel of $\gamma$, so corresponds to an orbit over $k$.
\end{prop}

\begin{proof}
By definition, the elements of the $2$-Selmer group $\Sel(J/k,2)$ correspond to classes in $H^1(k,J[2])$ whose restriction to $H^1(k_v,J[2])$ is in the image of $J(k_v)/2J(k_v)$ for every completion $k_v$. Hence the orthogonal space $U_v$ associated to the class $\gamma(\alpha_v)$ in $H^1(k_v,\SO(V))$ is split at every completion $k_v$. By the theorem of Hasse and Minkowski, a non-degenerate orthogonal space $U$ of dimension $2n+1$ is split over $k$ if and only if $U_v = U \otimes k_v$ is split over every completion $k_v$. Hence the orthogonal space $U$ associated to $\gamma(\alpha)$ is split over $k$, and $\alpha$ lies in the kernel of $\gamma$.
\end{proof}

%\begin{remark}{\em 
The same argument applies to the Selmer group of the Jacobian $J_d$ of the hyperelliptic curve $dy^2 = f(x)$, for any class $d \in k^*/k^{*2}$. Since the $2$-Selmer groups of the twisted curves are known to become arbitrarily large (cf. \cite{Bol} for the case of genus $n=1$),
%(finite, elementary abelian $2$-groups)
the number of $k$-rational orbits is infinite.
%}\end{remark}

\section{More general representations}

The three representations $V$ of $\SO(W)$ that we have studied illustrate various phenomena which occur in many other cases. For the standard representation, we have seen that the invariant polynomial $q_2$ distinguishes the orbits with $\Delta \neq 0$ over any field $k$. Here the arithmetic invariant theory is the same as the geometric invariant theory. 

This pleasant situation also occurs for orbits where the stabilizer $G_v$ is trivial! An interesting example for the odd orthogonal group $\SO(W)$ is the reducible representation $V = W \oplus \wedge^2(W)$. This occurs as the restriction of the adjoint representation of the split even orthogonal group of the space $W \oplus \langle -1 \rangle$. In this representation, the vector $v = (w,T)$ is stable if and only if the $2n+1$ vectors $\{w,T(w),T^2(w),\ldots,T^{2n}(w)\}$ form a basis of $W$, or equivalently, if the invariant polynomial $\Delta(v) = \det (\langle T^i(w), T^j(w)\rangle)$ is non-zero. In this case $G_v = 1$. 

One complication in this case is that the $k$-orbits do not cover the $k$-rational points of the categorical quotient: the map on points 
$$V(k) / \SO(W)(k) \rightarrow (V/\SO(W))(k)$$
is not surjective.  This situation is far more typical in invariant theory than the surjectivity for the three representations we studied.  Another atypical property of the three (faithful) representations we studied was that a generic vector had a nontrivial stabilizer.
%For general representations, the generic stabilizer of a vector is the kernel of the representation. 
%For representation theory, representations studied in the paper are strange because 
For a generic $v$ in a typical faithful representation $V$ of a reductive group $G$, the stabilizer $G_v$ is trivial.   For $G$ a torus and $k$ complex, $G_v$ is always the kernel of the representation; meanwhile, for $G$ simple, there are only finitely many exceptions (see~\cite[p.\ 229--235]{PV}). 
%It was a long project by Vinberg, V.L. Popov, V. Katz, Andreev, A.M. Popov, Elashvili to make this finite list. As a general statement, for k complex, dimGv is positive (as in first and second examples in this paper) if and only if dimV � dimG. , for the three representations we studied, every $k$-rational values of the invariant polynomials is realized by at least one $k$-orbit on $V$.

The adjoint representations $V = \frak g$ of split reductive groups $G$ generalize the second representation $V = \wedge^2(W) = \frak{so}(W)$. Here the invariant polynomials correspond to the invariants for the Weyl group on a Cartan subalgebra, and generate a polynomial ring of dimension equal to the rank of~$G$. The orbits where the discriminant $\Delta$ is non-zero correspond to the regular semi-simple elements in $\frak g$, and the stabilizer $G_v$ of such an orbit is a maximal torus in $G$. As an example, one can take the adjoint representation $V = \Sym^2(W)$ of the adjoint form $\PGSp(W) = \PGSp_{2n}$ of the symplectic group, where the degrees of the invariants are $2,4,6,\ldots,2n$. \cite[Ch 8, \S 13.3, VI] {B}. For some applications to knot theory, see~\cite{Miller}.

The representations which occur in Vinberg's theory for torsion automorphisms $\theta$ generalize the third representation $V = \Sym^2(W)$. Here the invariants again form a polynomial ring. As an example, one can take the reducible representation $\wedge^2(W)$ of the group $\PGSp(W) = \PGSp_{2n}$, which corresponds to the pinned outer involution $\theta$ of $\PGL_{2n}$. When $\theta$ lifts a regular elliptic class in the Weyl group, the orbits where the discriminant $\Delta$ is non-zero are stable, and the stabilizer $G_v$ is a finite commutative group scheme over $k$. Several examples of this type were discussed in \cite{G} and \cite{BH}.

\section{Integral orbits}

In order to develop a truly complete arithmetic invariant theory, we should consider orbits in representations not just over a field, but over $\Z$ or a general ring.  The descent from an algebraically closed field to a general field that we have discussed in Sections~2--8 gives an indication of some of the issues that arise over more general rings, and it serves as a useful guide for the more general integral theory.  In particular, just as a single orbit over an algebraically closed field can split into several orbits over a subfield, an orbit over say the field $\Q$ of rational numbers may then split into several orbits over $\Z$.  
%Much interesting and important information may be contained in this spliiting.

Often some of the most interesting arithmetic occurs in the passage from $\Q$ to $\Z$.
For example, consider the classical representation given by the action
of $\SL_2$ on binary quadratic forms $\Sym_2(2)$.  As we have already
noted, an orbit over $\Q$ (as over any field) is completely determined
by the value of the discriminant $d$ of the binary quadratic forms in
that orbit.  However, the set of primitive integral orbits inside the rational
orbit of discriminant $d\in \Z$ does not necessarily consist of one element, but rather
is in bijection with the set of (oriented) ideal classes of the quadratic order $\Z[(d+\sqrt{d})/2]$ in the quadratic field $\Q(\sqrt{d})$ (see, e.g.,~\cite{Buell}).
%The fact that the number of  is the famous (and as yet unsolved) class number one problem.

In general, to discuss integral orbits we must fix an
integral model of the representation being considered.  We give some
canonical integral models for the three representations we have
studied.  For the first representation, we take $W$ to be the odd
unimodular lattice of signature $(n+1,n)$ defined by (\ref{dprods}).
Because this lattice is self-dual, we can define the adjoint of an
endomorphism of $W$ over $\Z$.  The group $G$ is then the subgroup of
$\GL(W)$ consisting of those transformations $g$ such that $gg^\ast=1$
and $\det(g)=1$.  This defines a group that is smooth over $\Z[1/2]$
but is not smooth over $\Z_2$.  For the other representations of $G$,
we define $\wedge_2(W)$ as the lattice of skew self-adjoint
endomorphisms of $W$ equipped with the action of $G$ by conjugation;  we similarly define $\Sym_2(W)$ to be the
lattice of self-adjoint endomorphisms of~$W$. 
Our objective is to describe the orbits of $G$ on each of these three
$G$-modules, or at least those orbits where the discriminant invariant
is nonzero.  

\vspace{.05in}
For the standard representation $W$, we have already seen that there
is a unique orbit over $\Q$ for each value of the discriminant
$d\in\Q^\ast$. An invariant of a $\Z$-orbit of a vector $w$ in the
lattice $W$ with $\langle w,w \rangle = d$ is the isomorphism class
of the orthogonal complement $U = (\Z w)^{\perp}$, which is a lattice
of rank $2n$ and discriminant $d$ over $\Z$. Although $W$ is an odd lattice,
the lattice $U$ can be either even or odd. For example, when $n=3$, the orthogonal
complement $U$ of a primitive vector $w$ is an even bilinear space of rank
$2$ and discriminant $d$ (so corresponds to an integral binary quadratic form
of discriminant $d$) if and only
if the vector $w$
has the form $w = ae + bv + cf$ with $a$ and $c$ even and $b$ odd.
In this case $d = b^2 + 2ac \equiv 1$ modulo $8$, and the orbits of
$G(\Z)$ on such vectors form a principal homogeneous space for the
ideal class group of the quadratic order $\Z[(d + \sqrt{d})/2]$ of
discriminant $d$. We note that these are precisely the quadratic orders where the
prime $2$ is split. In this case the group $G(\Z)$ is 
isomorphic to the normalizer $N(\Gamma_0(2))$ of $\Gamma_0(2)$ in
$\PSL_2(\mathbb R)$, and the orbits described above correspond to 
the Heegner points of odd discriminant on the modular curve $X_0(2)^+$  \cite[\S1]{GKZ}.

\vspace{.05in}
We consider next the second representation $V=\wedge_2(W)$.
Here, we find that the integral orbits of $\SO(M)$ on the
self-adjoint transformations $T: M \to M$ with (separable)
characteristic polynomial $f(x)=xg(x^2)\in\Z[x]$ correspond to data which generalize
the notion of a ``minus ideal class'' for the ring $R=\Z[x]/(f(x))$.  
More precisely, the ring $R$ in $L=\Q[x]/(f(x))$ has an involution $\tau$ sending $\beta$ to $-\beta$, where $\beta$ denotes the image of $x$ in $R$. 
% This is because only terms of odd exponent appear in the
%polynomial $f(x)$ in $x$, so $f(\beta)=-f(-\beta)=0$.  
%Let $B$ denote the subring of $R$ fixed by $\tau$.  Then orbits on $\wedge^2(W)$ can
%then be classified by data which generalize the notion of ideal
%classes in $R$ that have trivial norm to $B$ (i.e., 
%{\it minus ideal classes}).
Let us consider pairs $(I, \alpha)$, where $I$ is a fractional ideal for $R$, the element $\alpha$ is in the $\Q$-subalgebra $F$ of $L$ fixed by $\tau$, the product $II^\tau$ is contained in the
principal ideal $(\alpha)$, and $N(I)N(I^\tau) = N(\alpha)$.
Such a pair $(I,\alpha)$ gives $I$  the structure of an integral lattice having
rank~$2n+1$ and determinant $(-1)^n$, where the symmetric bilinear
form on $I$ is defined by
%has the integral symmetric bilinear form
\begin{equation}\label{wedge2prod}
\langle x,y \rangle := \mbox{coefficient of $\beta^{2n}$ in 
  $(-1)^n\alpha^{-1}xy^\tau$}.
  \end{equation}
    The pair $(I',\alpha')$ gives an isometric lattice if $I' =
cI$ and $\alpha' = cc^\tau \alpha$ for some element $c\in L^*$.
%and we say that the pairs $(I,\alpha)$ and $(I',\alpha')$ are {\it $($Type I$)$ equivalent} in that case.
The operator $S: I \rightarrow I$ defined by $S(x) = \beta x$ is skew self-adjoint, and has characteristic polynomial $f(x)$. If the integral lattice determined by the pair $(I,\alpha)$ has signature $(n+1,n)$ over $\mathbb R$, there is an isometry $\theta: I \rightarrow M$ (cf. \cite{S2}), which is well-defined up to composition by an element in $\mathrm O(M)$. We obtain an $\SO(M)$-orbit of skew self-adjoint operators with characteristic polynomial $f(x)$ by taking $T = \theta S \theta^{-1}$. Conversely, since a
  skew self-adjoint $T: W \rightarrow W$ gives $W$ the structure of a
  torsion free $\Z[T] = R$-module of rank one, every integral orbit 
  %on $V=\wedge_2(W)$ 
  arises in this manner.
%Suppose that the algebra $L$ has $2m+1$ real embeddings $\eta_i: L \rightarrow \mathbb R$. Then the lattice determined by the pair $(I,\alpha)$ will have signature $(n+1,n)$ when exactly $m$ of the $2m+1$ non-zero real numbers $\eta_i(\alpha f'(\beta))$ are negative. 
Thus the 
equivalence classes of pairs $(I,\alpha)$ for the ring $R=\Z[x]/(f(x))$, as defined above,
%up to the equivalence relation defined by $c$ in $L^*$ 
%form a subset of the group $A$ and 
index the finite number of integral orbits on $V=\wedge_2(W)$ with characteristic polynomial $f(x)$.

%When $R$ is maximal, the equivalence classes of pairs $(I,\alpha)$ form a finite abelian group $A$, which is related to the {\it minus class group} of $R$, i.e., the subgroup of elements $I$ in the class group of $R$ having trivial norm to the class group of $B$, where $B$ is the subring of $R$ fixed by the involution $\tau$.  The group $A$ is related to the minus class group $\Pic^-(R)$ of $R$ 
% by the exact sequence
%$$ 1 \rightarrow (R^*/N^R_B(R^*))_{N=1} \rightarrow A \rightarrow \Pic^-(R) \rightarrow 1.$$  

%Namely, we say that the pair $(I',\alpha')$ is {\it equivalent} to $(I,\alpha)$ if there is an element $c\in L^*$ with $I' = cI$ and $\alpha' = c^2\alpha$.  
%The equivalence classes of pairs $(I,\alpha)$ then classify the orbits
%of $G(\Z)$ on $V(\Z)$, as follows:

%We may similarly treat the third representation $\Sym_2(W)$.  
Let us now consider the third representation $V=\Sym_2(W)$.  
When $\dim (W) = 3$, the kernel of the trace map gives a lattice of rank $5$, closely
related to the space of binary quartic forms for $\PGL_2$. The
integral orbits in this case were studied in \cite{B1} and
\cite{W}. In general, the integral orbits of $\SO(M)$ on the
self-adjoint transformations $T: M \to M$ with (separable)
characteristic polynomial $f(x)$ correspond to data which generalize
the notion of an ideal class of order $2$ for the order $R = \mathbb
Z[x]/(f(x))$ in the $\Q$-algebra $L=\mathbb Q[x]/(f(x))$. 
More precisely,
%we define an equivalence relation on 
we consider pairs $(I, \alpha)$, where $I$ is a fractional ideal for
$R$, the element $\alpha$ lies in $L^*$, the square $I^2$ of the ideal
$I$ is contained in the principal ideal $(\alpha)$, and the square of
the norm of $I$ satisfies $N(I)^2 = N(\alpha)$.  Then the lattice $I$ has the integral symmetric bilinear form
\begin{equation}\label{sym2prod}
\langle x,y \rangle := \mbox{coefficient of $\beta^{2n}$ in $\alpha^{-1}xy$}
\end{equation}
of determinant $(-1)^n$, and self-adjoint operator given by multiplication by $\beta$, where $\beta$ again denotes the image of $x$ in $R$. The pair $(I',\alpha')$ gives an isometric lattice if $I' = cI$ and $\alpha' = c^2 \alpha$ for some element $c \in L^*$.
%and we say that the pairs $(I,\alpha)$ and $(I',\alpha')$ are {\it equivalent} in that case. 
When this lattice has signature $(n+1,n)$ over $\mathbb R$, it is isometric to $M$ and we obtain an integral orbit with characteristic polynomial $f(x)$.  Conversely, since a self-adjoint $T: W \rightarrow W$ gives~$W$ the structure of a  torsion free $\Z[T] = R$-module of rank one, every integral orbit arises in this way.
Thus 
%equivalence classes of 
pairs $(I,\alpha)$ for the ring $R=\Z[x]/(f(x))$,
up to the equivalence relation defined by $c$ in $L^*$,
%form a subset of the group $A$ and 
index the finite number of integral orbits on $V=\Sym_2(W)$ with characteristic polynomial~$f(x)$.

%The equivalence classes of pairs $(I,\alpha)$ for the ring $R$, as defined above, form a finite abelian group $A$, which is related to the group of ideal classes of order $2$ for $R$ by the exact sequence
%$$ 1 \rightarrow (R^*/R^{*2})_{N=1} \rightarrow A \rightarrow \Pic(R)[2] \rightarrow 1.$$
%Suppose that the algebra $L$ has $2m
%1$ real embeddings $\eta_i: L \rightarrow \mathbb R$. Then the lattice determined by the pair $(I,\alpha)$ 
%will have signature $(n+1,n)$ when exactly $m$ of the $2m+1$ non-zero real numbers $\eta_i(\alpha 
%'(\beta))$ are negative. 
%The pairs $(I,\alpha)$, up to the equivalence relation defined by $c$ in $L^*$, 
%form a subset of the group $A$ and index the finite number of integral orbits with characteristic polynomial $f(x)$.

%The equivalence classes of such pairs $(I,\alpha)$
%Namely, we say that the pair $(I',\alpha')$ is {\it equivalent} to $(I,\alpha)$ if there is an element $c\in L^*$ with $I' = cI$ and $\alpha' = c^2\alpha$.  
%The equivalence classes of pairs $(I,\alpha)$ can be used to classify the orbits of $G(\Z)$ on $V(\Z)$, as follows:
%We may classify the orbits of $G(\Z)$ on $V(\Z)$ in terms of equivalence classes of pairs $(I,\alpha)$,
%as follows:

We summarize what we have established for the representations $V=\wedge_2(W)$ and $V=\Sym_2(W)$.

\begin{proposition}\label{int2}
  Let $V$ denote either the representation $\wedge_2(W)$ or $\Sym_2(W)$ of $G$.  Let $f(x)$ be a polynomial of degree $2n+1$ with coefficients in
  $\Z$ and non-zero discriminant in $\Q$;  if $V=\wedge_2(W)$ we further assume that $f(x)=xg(x^2)$ for an integral polynomial $g$.  Then the integral orbits of
  $G(\Z)$ on $V(\Z)$
%  self-adjoint operators $T$ 
  with characteristic polynomial
  $f(x)$ are in bijection with the equivalence classes of pairs $(I, \alpha)$ for the order $R = \Z[x]/(f(x))$ defined above,
%$($as defined above for each possibility for $V)$
with the property that the
  bilinear form $\langle \;,\,\rangle$ on $I$ $($given by $(\ref{wedge2prod})$ or $(\ref{sym2prod})$, respectively) is split. 
 \end{proposition}
 
  In terms of Proposition~$\ref{rationalws}$, the integral orbit corresponding to the pair $(I,\alpha)$ maps to the rational  orbit of $\SO(W)(\Q)$ on $V(\Q)$ corresponding to the class of $\alpha \equiv\alpha^{-1}$. Here we view $\alpha$ as an element of 
  $(K^*/NE^*)$ when  $V=\wedge_2(W)$, so $L = E + \Q$ and $L^{\tau} =  K + \Q$. When $V = \Sym_2(W)$, we view $\alpha$ as an element of  $(L^*/L^{*2})_{N \equiv1}$.

  Finally, we remark that it would be interesting and useful to
  develop a theory of cohomology that allows one to describe
  orbits over the integers as we have in the cases above.
  For example, let us consider again the representation $V$ of the
  group $G = \PGL_2$ over $ \mathbb Q$ given by conjugation on the
  $2\times 2$ matrices $v$ of trace zero. Then this is the adjoint
  representation, and is also the standard representation of $\SO_3
  \cong \PGL_2$. The ring of invariant polynomials on $V$ is generated
  by $q(v):= -\det(v)$, and the stabilizer $G_v$ of a vector with
  $q(v) = d \neq 0$ is isomorphic to the one-dimensional torus over
  $\mathbb Q$ which is split by $K = \mathbb Q(\sqrt d)$, and all
  vectors $w$ with $q(w) = q(v) \neq 0$ lie in the same $G(\mathbb
  Q)$-orbit.

  A natural integral model of this representation is given by the action of the
  $\mathbb Z$-group $G = \PGL_2$ on the finite free $\mathbb Z$-module
  of binary quadratic forms $ax^2 + bxy + cy^2$. This is equivalent to
  the representation by conjugation on the matrices of trace zero in
  the subring $\mathbb Z + 2\mathcal R$ of the ring $\mathcal R$ of
  $2\times2$~integral matrices. In this model, the invariant
  polynomial is just the discriminant $ d = b^2 - 4ac$ of the binary
  form. The content $e = \gcd(a,b,c)$ is also an invariant of a
  non-zero integral orbit.

  We may calculate the $G(\mathbb Z)$-orbits on the set $S$ of forms
  with discriminant $d \in \mathbb Z - \{0\}$ and content $e = 1$ (so
  the binary quadratic form is primitive) via cohomology.  Let $O =
  O(d)$ be the quadratic order of discriminant $d$. Then the
  stabilizer $G_v$ of such an orbit in $\PGL_2$ is a smooth group
  scheme over $\mathbb Z$ which lies in an exact sequence (in the
  \'etale topology)
$$1 \rightarrow \mathbb G_m \rightarrow  \Res_{O/\mathbb Z} \mathbb G_m \rightarrow G_v \rightarrow 1.$$
Furthermore, the $\mathbb Z$-points of the quotient scheme $G/G_v$ can
be identified with the set $S$. Hence the orbits in question are in
bijection with the kernel of the map $\gamma: H^1(\mathbb Z, G_v)
\rightarrow H^1(\mathbb Z, \PGL_2)$ in \'etale cohomology. Since
$H^1(\mathbb Z, \PGL_2) = 1$, the orbits are in bijection with the
elements of $H^1(\mathbb Z,G_v)$.  Since $H^1(\mathbb Z, \mathbb G_m)
= H^2(\mathbb Z, \mathbb G_m) = 1$ the long exact sequence in
cohomology gives
$$H^1(\mathbb Z,G_v) = H^1(\mathbb Z, \Res_{O/\mathbb Z} \mathbb G_m) = \Pic(O).$$
Hence the orbits of $\PGL_2(\mathbb Z)$ on the set $S$ of binary
quadratic forms of discriminant $d \neq 0$ and content~$1$ form a
principal homogeneous space for the finite group $\Pic(O(d))$ of
isomorphism classes of projective $O(d)$-modules of rank one. Thus the
number of primitive integral orbits contained in the rational orbit of
discriminant $d$ is given by the class number of $O(d)$.
%Since this group can be non-trivial, there may be several integral orbits which are contained in the same rational orbit.

%The representation $\Sym^2(W)$ of $\SO(W)$ has an integral model consisting of the self-adjoint endomorphisms of the lattice $M$ of  $W$ spanned by the standard basis elements. 

\def\noopsort#1{}
\providecommand{\bysame}{\leavevmode\hbox to3em{\hrulefill}\thinspace}


\begin{thebibliography}{10}

\bibitem{BH} M.\ Bhargava and W.\ Ho, Coregular spaces and genus one
  curves, preprint.

%\bibitem{BG}
%M.\ Bhargava and B.\ Gross, The average size of the 2-Selmer group of Jacobians of hyperelliptic curves having a rational Weierstrass point, preprint.

\bibitem{B1}
M.\ Bhargava and A.\ Shankar,
{Binary quartic forms having bounded invariants, and the boundedness of the average rank of elliptic curves},
ArXiv: 1006.1002 (2010).

\bibitem{B}
N.\ Bourbaki, \emph{Groupes et alg\`ebres de Lie}, Hermann, 1982.

\bibitem{BSD}
 B.\ J.\ Birch and H.\ P.\ F.\ Swinnerton-Dyer, Notes on elliptic curves I,  
{\it J. Reine Angew. Math.} {\bf 212} (1963), 7--25.

\bibitem{Bol}
R.\ Bolling, Die Ordnung der Schafarewitsch-Tate-Gruppe kann beliebig gross  werden,
{\it Math. Nachr.} {\bf 67} (1975), 157--179.

\bibitem{Buell}
D.\ A.\ Buell, {\it Binary quadratic forms: classical theory and modern computations}, Springer-Verlag, 1989.
%\bibitem{Cremona}
%J.\ Cremona,

\bibitem{D}
R.~Donagi,
{Group law on the intersection of two quadrics},
{\it Annali della Scuola Normale Superiore di Pisa} {\bf 7} (1980), 217--239.


\bibitem{G}
B.~Gross, {On Bhargava's representations and Vinberg's invariant
  theory}, In: \emph{Frontiers of Mathematical Sciences}, International Press (2011), 317--321.


\bibitem{GKZ}
B.\ Gross, W.\ Kohnen, and D.\ Zagier, Heegner points and derivatives of $L$-series II, {\it Math.\ Ann.}~{\bf 278} (1987), 497--562.


\bibitem{K1} R.~Kottwitz, 
{Stable trace formula: cuspidal tempered terms},
{\it Duke Math.\ J.} \textbf{51} (1984), 611--650.

\bibitem{K}
A.\ Knus, A.\ Merkurjev, M.\ Rost, and J.-P.\ Tignol,
\emph{The book of involutions},
AMS Colloquium Publications \textbf{44}, 1998.

\bibitem{KT} A.\ Kostrikin and P.\ H.\ Tiep, {\it Orthogonal
    decompositions and integral lattices}, deGruyter Expositions in
   Mathematics {\bf 15}, Berlin, 1994.

\bibitem{L} R.\ Langlands,  {Stable conjugacy---definitions and lemmas},
{\it Canadian J.\ Math} \textbf{31} (1979), 700--725.

\bibitem{L2} R.\ Langlands, {Les d\'ebuts d'une formule des traces stable},
{\it Publ.\ Math.\ de L'Univ.\ Paris VII}, \textbf{13}, 1983.


\bibitem{Miller}
A.\ Miller, Knots and arithmetic invariant theory, preprint.

\bibitem{M}
J.\ Milnor and D.\ Husemoller,
\emph{Symmetric bilinear forms},
Springer Ergebnisse \textbf{73}, 1970.


\bibitem{Mu}
D.\ Mumford, J.\ Fogarty, F.\ Kirwan
{\it Geometric invariant theory},
Springer Ergebnisse {\bf 34}, 1994.


\bibitem{P}
D.~Panyushev,
{On invariant theory of $\theta$-groups,}
{\it J.\ Algebra} \textbf{283} (2005), 655--670.

\bibitem{PR}
V.\ Platonov and A.\ Rapinchuk, {\it Algebraic groups and
  number theory}, Translated from the 1991 Russian original by Rachel
  Rowen, {\it Pure and Applied Mathematics} {\bf 139}, Academic Press, Inc.,
  Boston, MA, 1994.

\bibitem{PV}
V.\ L.\ Popov and E.\ B.\ Vinberg, {\it Invariant Theory}, in {\it Algebraic Geometry IV}, Encylopaedia of Mathematical Sciences {\bf 55}, Springer-Verlag, 1994.

\bibitem{S}
J-P.~Serre,
\emph{Galois cohomology}, Springer Monographs in Mathematics, 2002.

\bibitem{S2}
J-P.~Serre,
\emph{A course in arithmetic},
Springer GTM \textbf{7}, (1978).

\bibitem{Sh}
D.~Shelstad
{Orbital integrals and a family of groups attached to a real reductive group},
{\it Ann. Sci. \'Ecole Norm. Sup.} {\bf 12} (1979), 1--31.

\bibitem{Stoll}
M.~Stoll, {Implementing $2$-descent for Jacobians of hyperelliptic curves},
{\it Acta Arith} {\bf 98} (2001), 245--277.

\bibitem{W} M.~Wood, {\it Moduli spaces for rings and ideals}, Ph.D.\
  thesis, Princeton University, 2008.

\end{thebibliography}
\end{document}